\documentclass[a4paper]{amsart}
\usepackage{graphicx}
\usepackage{amssymb}
\usepackage{amsmath}
\usepackage{amsthm}
\usepackage{amscd}

  \usepackage{fancyhdr}

\usepackage[all,2cell]{xy}

\UseAllTwocells \SilentMatrices
\newtheorem{thm}{Theorem}[section]

\newtheorem{cor}[thm]{Corollary}
\newtheorem{lem}[thm]{Lemma}
\newtheorem{exm}[thm]{Example}

\newtheorem{prop}[thm]{Proposition}
\theoremstyle{definition}

\theoremstyle{remark}
\newtheorem{rem}[thm]{\bf Remark}
\numberwithin{equation}{section}

\begin{document}

\title[Hereditary triangulated categories]{Hereditary triangulated categories}

\author[Xiao-Wu Chen, Claus Michael Ringel] {Xiao-Wu Chen, Claus Michael Ringel}

\date{\today}


\maketitle

\dedicatory{}%
\commby{}%

\begin{abstract}
We call a triangulated category  \emph{hereditary} provided that it is equivalent to the bounded derived category of a hereditary abelian category, where the equivalence is required to commute with the translation functors. If the triangulated category is algebraical, we may replace the equivalence by a triangle equivalence. We give two intrinsic characterizations of hereditary triangulated categories using a certain full subcategory and the non-existence of certain paths. We apply them to piecewise hereditary algebras.
\end{abstract}

\section{Introduction}

Hereditary abelian categories play a central role in the representation theory of finite dimensional algebras. They are related to piecewise hereditary algebras, an important class of algebras. If the ground field is algebraically closed and the hereditary abelian category has a tilting object, then up to derived equivalence, it is the module category of a path algebra or the category of coherent sheaves on a weighted projective line; see \cite{Hap2}.

We aim to characterize the bounded derived category of a hereditary abelian category among arbitrary triangulated categories. These triangulated categories should be called hereditary. More precisely, we call a triangulated category $\mathcal{D}$ \emph{hereditary} provided that there is an equivalence $F$ between $\mathcal{D}$ and the bounded derived category of a hereditary abelian category, where $F$ is required to commute with the translation functors. A prior the equivalence $F$ may not be a triangle equivalence. However, if the triangulated category $\mathcal{D}$ is \emph{algebraical}, that is, triangle equivalent to the stable category of a Frobenius category, we can replace the equivalence $F$ by a triangle equivalence.

The main results are two intrinsic characterizations of  hereditary triangulated categories; see Theorems \ref{thm:A} and \ref{thm:B}: one uses a certain full subcategory  in the triangulated category,  and the other uses the non-existence of a certain path in the triangulated category. These results  give new characterizations to piecewise hereditary algebras.

The paper is structured as follows. In Section 2, we characterize hereditary triangulated categories using hereditary $t$-structures and prove Theorem \ref{thm:A}. In Section 3, we prove that if the given triangulated category is algebraical, then the equivalence $F$ mentioned above might be replaced by a triangle equivalence; see Theorem \ref{thm:A.1}. This relies on an existence result in \cite{Bei} on the realization functor for a given $t$-structure.  In Section 4, we study paths in a triangulated category. If the category is a block, that is, indecomposable as a triangulated category, the existence of a certain path is proved in Proposition \ref{prop:path}.  In Section 5, we prove Theorem \ref{thm:B}. We give some applications to piecewise hereditary algebras in the end.

\section{Hereditary triangulated categories}

In this section, we give various characterizations to hereditary triangulated categories. In particular, a triangulated category is hereditary if and only if it has a hereditary $t$-structure.

\subsection{Hereditary $t$-structures}

Let $\mathcal{D}$ be a triangulated category with its translation functor denoted by $[1]$. We denote by $[-1]$ the quasi-inverse of $[1]$.  For two full subcategories $\mathcal{U}$ and $\mathcal{V}$, we denote by $\mathcal{U}\star\mathcal{V}$  the full subcategory consisting of those objects $X$ that fit into an exact triangle $U\rightarrow X\rightarrow V\rightarrow U[1]$ with $U\in \mathcal{U}$ and $V\in \mathcal{V}$. The operation  $\star$ is associative; see \cite[Subsection 1.3.9]{BBD}.

 Recall from \cite[Section 1.3]{BBD} that a {\em $t$-structure} on $\mathcal{D}$ is a pair $(\mathcal{D}^{\leq
0},
\mathcal{D}^{\geq 0})$ of full additive subcategories  satisfying the following conditions:
\begin{enumerate}
\item[(T1)] ${\rm Hom}_\mathcal{D}(X, Y[-1])=0$ for all $X \in
\mathcal{D}^{\leq 0}$ and $Y \in \mathcal{D}^{\geq 0}$;
\item[(T2)] $\mathcal{D}^{\leq 0}$ is closed under $[1]$, and
$\mathcal{D}^{\geq 0}$ is closed under $[-1]$;
\item[(T3)] For each $X \in \mathcal{D}$, there is an exact triangle
$A \rightarrow X \rightarrow B[-1] \rightarrow A[1]$ with $A\in
\mathcal{D}^{\leq 0}$ and $B\in \mathcal{D}^{\geq 0}$.
\end{enumerate}
Set $\mathcal{A}=\mathcal{D}^{\leq 0}\cap\mathcal{D}^{\geq 0}$ to be the
{\em heart} of the $t$-structure, which is
an abelian category. Moreover, a sequence $\xi\colon 0\rightarrow X \stackrel{f}\rightarrow Y\stackrel{g}\rightarrow Z\rightarrow 0$ in $\mathcal{A}$  is exact if and only if there is an exact triangle  $X \stackrel{f}\rightarrow Y\stackrel{g}\rightarrow Z\stackrel{\omega} \rightarrow X[1]$ in $\mathcal{D}$. Indeed, the triangle is unique, since such a morphism $\omega$ is uniquely determined by $f$ and $g$. Then we have an induced isomorphism
\begin{align}\label{equ:iso1}
{\rm Ext}_\mathcal{A}^1(Z, X)\longrightarrow {\rm Hom}_\mathcal{D}(Z, X[1]), \quad [\xi]\mapsto \omega.
\end{align}
Here, $[\xi]$ denotes the equivalence class of $\xi$ in ${\rm Ext}_\mathcal{A}^1(Z, X)$. We refer to \cite[Th\'{e}or\`{e}me 1.3.6]{BBD} for details.

Let $(\mathcal{D}^{\leq 0}, \mathcal{D}^{\geq 0})$ be a $t$-structure on $\mathcal{D}$. Set $\mathcal{D}^{\leq n}=\mathcal{D}^{\leq 0}[-n]$ and
$\mathcal{D}^{\geq n}=\mathcal{D}^{\geq 0}[-n]$, $n \in \mathbb{Z}$.
Recall that the {\em truncation
functors} $\tau_{\leq 0}: \mathcal{D} \longrightarrow
\mathcal{D}^{\leq 0}$ and $\tau_{\geq 1}: \mathcal{D}
\longrightarrow \mathcal{D}^{\geq
1}$ are the right adjoint and the left adjoint of the inclusion functors ${\rm inc}\colon \mathcal{D}^{\leq 0}\rightarrow \mathcal{D}$ and
${\rm inc}\colon \mathcal{D}^{\geq 1}\rightarrow \mathcal{D}$, respectively.

In general, one defines $\tau_{\leq n} :\mathcal{D} \longrightarrow
\mathcal{D}^{\leq n}$ and $\tau_{\geq n+1}: \mathcal{D}
\longrightarrow \mathcal{D}^{\geq n+1} $ by $\tau_{\leq n}=
[-n]\circ \tau_{\leq 0}\circ [n]$ and $\tau_{\geq n+1}=[-n]\circ
\tau_{\geq 1} \circ [n]$, respectively. Then $\tau_{\leq n}$ and $\tau_{\geq n+1}$ coincide with the truncation functors associated to the shifted $t$-structure $(\mathcal{D}^{\leq n}, \mathcal{D}^{\geq n})$.  In particular, for each object $X$, the triangle in (T3) yields an exact triangle
\begin{align}\label{equ:tri1}
\tau_{\leq n}X \longrightarrow X \longrightarrow \tau_{\geq n+1}X \stackrel{c}\longrightarrow
(\tau_{\leq n}X)[1],
\end{align}
where the morphisms yield natural transformations between functors.  The $n$-th \emph{cohomological functor} $H^n\colon \mathcal{D}\rightarrow \mathcal{A}$ is defined to be
$H^n(X)=(\tau_{\geq n}\tau_{\leq n})(X)[n]$. We observe that $H^n(X)\simeq H^0(X[n])$.

The $t$-structure $(\mathcal{D}^{\leq 0},
\mathcal{D}^{\geq 0})$  is called {\em bounded}, if  for each $X \in \mathcal{D}$, there
exist $m \leq n$ such that $X \in \mathcal{D}^{\leq n}\cap
\mathcal{D}^{\geq m}$. We observe that an object $X$  lies in $\mathcal{D}^{\leq n}\cap
\mathcal{D}^{\geq m}$ if and only if $H^p(X)=0$ for $p<m$ or $p>n$. Moreover, we have
\begin{align}\label{equ:iso4}
\mathcal{D}^{\leq n}\cap \mathcal{D}^{\geq m}=\mathcal{A}[-m]\star \mathcal{A}[-(m+1)]\star \cdots \star \mathcal{A}[-n].
\end{align}

 Inspired by \cite[Section 6]{Kel05}, we call  a bounded $t$-structure $(\mathcal{D}^{\leq 0},
\mathcal{D}^{\geq 0})$ \emph{hereditary}, provided that ${\rm Hom}_\mathcal{D}(X, Y[n])=0$ for $n\geq 2$ and $X, Y\in \mathcal{A}$. Hereditary $t$-structures are called \emph{split} in \cite[Definition 4.1]{BR}.

The following result is essentially due to \cite[Proposition 1 b)]{Kel05}.

 \begin{lem}\label{lem:split}
 Let $(\mathcal{D}^{\leq 0},
\mathcal{D}^{\geq 0})$ be a hereditary $t$-structure on $\mathcal{D}$ with its heart $\mathcal{A}$. Then the abelian category $\mathcal{A}$ is hereditary and each object $X\in \mathcal{D}$ is isomorphic to $\bigoplus_{p\in \mathbb{Z}} H^p(X)[-p]$.
 \end{lem}

 \begin{proof}
 Indeed, by the isomorphism (\ref{equ:iso1}) the functor ${\rm Ext}_\mathcal{A}^1(Z, -)$ is right exact for each object $Z$ in $\mathcal{A}$. This implies that the abelian category $\mathcal{A}$ is hereditary.

 We observe that an object $A$ in $\mathcal{D}^{\geq n+1}$ necessarily lies in  $\mathcal{D}^{\leq r}\cap\mathcal{D}^{\geq n+1}$ for some $r>n+1$. Similarly, an object $B$ in $\mathcal{D}^{\leq n-1}$ lies in $\mathcal{D}^{\leq n-1}\cap \mathcal{D}^{\geq s}$ for  some $s<n-1$. By (\ref{equ:iso4}) and the hereditary assumption, we infer that ${\rm Hom}_\mathcal{D}(A, B)=0$. In particular, the morphism $c$ in (\ref{equ:tri1}) vanishes. Then the triangle (\ref{equ:tri1}) splits and thus $X\simeq \tau_{\leq n}X \oplus \tau_{\geq n+1}X$.

 To prove the last statement,  we use induction on $l(X)$, the cardinality of the set $\{p\in \mathbb{Z}\; |\; H^p(X)\neq 0\}$. If $l(X)=1$, we assume that $H^p(X)\neq 0$. Then $X$ lies in $\mathcal{D}^{\leq p} \cap \mathcal{D}^{\geq p}$. It follows that $X\simeq H^p(X)[-p]$. In general, we take the largest $p$ with $H^p(X)\neq 0$. We observe that $l(\tau_{\leq p-1}X)=l(X)-1$ and $l(\tau_{\geq p}X)=1$. Applying the induction, we are done by the isomorphism $X\simeq \tau_{\leq p-1}X \oplus \tau_{\geq p}X$.
\
\end{proof}

The canonical example is as follows.

\begin{exm}\label{exm:can}
{\rm Let $\mathcal{A}$ be an abelian category. The bounded derived category $\mathbf{D}^b(\mathcal{A})$ has a \emph{canonical $t$-structure} with $\mathbf{D}^b(\mathcal{A})^{\leq 0}=\{X\in \mathbf{D}^b(\mathcal{A})\; |\; H^i(X)=0 \mbox{ for }i>0\}$ and $\mathbf{D}^b(\mathcal{A})^{\geq 0}=\{X\in \mathbf{D}^b(\mathcal{A})\; |\; H^i(X)=0 \mbox{ for }i<0\}$. Here, $H^i$ denotes the $i$-th cohomology of a complex. The heart is naturally identified with $\mathcal{A}$. Here, the abelian category $\mathcal{A}$ is canonically embedded into $\mathbf{D}^b(\mathcal{A})$ by sending each object $A$ to the stalk complex concentrated on degree zero, which is still denoted by $A$.

The canonical $t$-structure is bounded. Moreover, it is hereditary if and only if the category $\mathcal{A}$ is hereditary. In this case,  each object $X$ in $\mathbf{D}^b(\mathcal{A})$ is isomorphic to $\bigoplus_{i\in \mathbb{Z}} H^p(X)[-p]$ by Lemma \ref{lem:split}; compare \cite[Subsection 1.6]{Kra}. }
\end{exm}

\subsection{Characterizations of hereditary triangulated categories}

Let $\mathcal{D}$ be a triangulated category. For a full subcategory $\mathcal{S}$, we denote by ${\rm add}\; \mathcal{S}$ the smallest additive subcategory containing $\mathcal{S}$ and closed under isomorphisms. We do not require that ${\rm add}\; \mathcal{S}$ is closed under direct summands.  Let $\mathcal{S}'$ be another full subcategory. By ${\rm Hom}_\mathcal{D}(\mathcal{S}, \mathcal{S}')=0$, we mean that ${\rm Hom}_\mathcal{D}(X, Y)=0$ for each object $X\in \mathcal{S}$ and $Y\in \mathcal{S}'$.

\begin{thm}\label{thm:A}
Let $\mathcal{D}$ be a triangulated category with $\mathcal{A}$ its full additive subcategory. The following statements are equivalent:
\begin{enumerate}
\item There is a hereditary $t$-structure on $\mathcal{D}$ with $\mathcal{A}$ its heart.
\item $\mathcal{D}={\rm add}\; (\bigcup_{n\in \mathbb{Z}}\mathcal{A}[n])$ and ${\rm Hom}_\mathcal{D}(\mathcal{A}, \mathcal{A}[m])=0$ for $m<0$.
\item The category $\mathcal{A}$ is hereditary abelian with an equivalence $F\colon \mathbf{D}^b(\mathcal{A})\rightarrow \mathcal{D}$ of categories, which commutes with the translation functors and respects the canonical embedding of $\mathcal{A}$ into $\mathbf{D}^b(\mathcal{A})$.
\end{enumerate}
\end{thm}

We will call a triangulated category $\mathcal{D}$ \emph{hereditary} provided that there is a full additive subcategory $\mathcal{A}$ satisfying one of the above equivalent conditions.

\begin{proof}
The implication ``$(1)\Rightarrow (2)$" follows from Lemma \ref{lem:split}. Example \ref{exm:can} implies ``$(3)\Rightarrow (1)$".

For ``$(2)\Rightarrow (1)$", we observe first that any object $X \in \mathcal{A}\cap (\mathcal{A}[n])$ is necessarily zero for nonzero $n$. Set $\mathcal{D}^{\leq 0}={\rm add}\; (\bigcup_{n\geq 0}\mathcal{A}[n])$ and $\mathcal{D}^{\geq 0}={\rm add}\; (\bigcup_{n\leq 0}\mathcal{A}[n])$. We claim that $(\mathcal{D}^{\leq 0}, \mathcal{D}^{\geq 0})$ is a bounded $t$-structure on $\mathcal{D}$. Indeed, the conditions (T1) and (T2) are immediate. Take any object $X\in \mathcal{D}$. By the assumption, we have $X=A\oplus (B[-1])$ with $A\in \mathcal{D}^{\leq 0}$ and $B\in \mathcal{D}^{\geq 0}$. Then the split triangle $A\rightarrow X\rightarrow B[-1]\rightarrow A[1]$ proves (T3). The boundedness of this $t$-structure is evident.

The heart of the above $t$-structure is $\mathcal{A}$. To prove that the $t$-structure is hereditary, we take a morphism $u\colon A\rightarrow B[n]$ with $A, B\in \mathcal{A}$ and $n\geq 2$. Form an exact triangle $A\stackrel{u}\rightarrow B[n]\rightarrow C_1\oplus C_2\rightarrow A[1]$ with $C_1\in {\rm add}\; (\bigcup_{m\geq 2}\mathcal{A}[m])$ and $C_2\in {\rm add}\; (\bigcup_{m\leq 1}\mathcal{A}[m])$. We observe that ${\rm Hom}_\mathcal{D}(B[n], C_2)=0={\rm Hom}_\mathcal{D}(C_1, A[1])$. Then we have $u=0$ by Lemma \ref{lem:basic}(3).

It remains to show ``$(1)+(2)\Rightarrow (3)$". By Lemma \ref{lem:split}, the category $\mathcal{A}$ is hereditary abelian. In particular, we  have $\mathbf{D}^b(\mathcal{A})={\rm add}\; \; (\bigcup_{n\in \mathbb{Z}}\mathcal{A}[n])$. We construct an additive functor $F\colon \mathbf{D}^b(\mathcal{A})\rightarrow \mathcal{D}$ as follows. For each $n\in \mathbb{Z}$, we set  $F(A[n])=A[n]$ and $F(f[n])=f[n]$ for any object $A\in \mathcal{A}$ and any morphism $f\colon A\rightarrow B$ in $\mathcal{A}$. For a morphism $w\in {\rm Hom}_{\mathbf{D}^b(\mathcal{A})}(A, B[1])$, we consider the exact triangle $B\stackrel{a}\rightarrow E\stackrel{b}\rightarrow A \stackrel{w}\rightarrow B[1]$ in $\mathbf{D}^b(\mathcal{A})$, where $0\rightarrow  B\stackrel{a}\rightarrow E\stackrel{b}\rightarrow A\rightarrow 0$ is the short exact sequence corresponding to $w$. We define the morphism $F(w)$ by the unique exact triangle $B\stackrel{a}\rightarrow E\stackrel{b}\rightarrow A \stackrel{F(w)}\rightarrow B[1]$ in $\mathcal{D}$. More generally, we set $F(w[n])=F(w)[n]$. One verifies that $F$ is indeed a functor, where the bifunctorialness of the isomorphism (\ref{equ:iso1}) is implicitly used. Then this functor $F$ is as required.
\end{proof}

The following fact is standard.

\begin{lem}\label{lem:basic}
Let $A\stackrel{u}\rightarrow B  \stackrel{\binom{v}{0}}\rightarrow C_1\oplus C_2\rightarrow A[1]$ be an exact triangle in $\mathcal{D}$.  Then the following statements hold.
\begin{enumerate}
\item The object $C_2$ is a direct summand of $A[1]$. In particular, $C_2=0$ whenever ${\rm Hom}_\mathcal{D}(C_2, A[1])=0$.
    \item If $A$ is indecomposable and $C_2\neq 0$, then we have  $u=0$.
\item If ${\rm Hom}_\mathcal{D}(C_1, A[1])=0$ and ${\rm Hom}_\mathcal{D}(B, A)=0$, then we have $u=0$.
\end{enumerate}
\end{lem}

\begin{proof}
The morphism $B\rightarrow C_1\oplus C_2$ is of the form $v\oplus 0\colon B\oplus 0\rightarrow C_1\oplus C_2$.  It follows that the given triangle is isomorphic to the direct sum of $A'\rightarrow B\stackrel{v}\rightarrow C_1\rightarrow A'[1]$ and $C_2[-1]\rightarrow 0\rightarrow C_2\stackrel{{\rm Id}} \rightarrow C_2$. Then (1) follows immediately. For (2), we just observe that $A'=0$.

For (3), we observe that ${\rm Hom}_\mathcal{D}(C_1, A'[1])=0$ since $A'$ is a direct summand of $A$. For the same reason, we have ${\rm Hom}_\mathcal{D}(B, A')=0$. However, by \cite[Lemma I.1.4]{Hap} the morphism $A'\rightarrow B$ is split mono, which is then forced to be zero. Then we are done.
\end{proof}

For a finite dimensional algebra $A$ over a field, we denote by $A\mbox{-mod}$ the abelian category of finite dimensional left $A$-modules.

\begin{exm}
{\rm There does exist a full additive subcategory $\mathcal{A}$
of a triangulated category $\mathcal{D}$ such that $\mathcal{A}$ is hereditary abelian with $\mathcal{D}={\rm add}\; (\bigcup_{n\in \mathbb{Z}}\mathcal{A}[n])$, whereas the condition ${\rm Hom}_\mathcal{D}(\mathcal{A}, \mathcal{A}[m])=0$ for $m<0$ is not satisfied.

Namely, let $\mathcal{D}=\mathbf{D}^b(A\mbox{-mod})$ with $A$  the path algebra of a quiver of type $A_2$ over a field $k$. This is the quiver with two vertices, say $1$ and $2$, and a single arrow
$1\rightarrow  2$. Note that $A\mbox{-mod}$ has precisely three indecomposable modules, say
$S_1, I, S_2$, where $S_1$ is simple injective, $I$ has length 2, and $S_2$ is simple projective.
Consider the full subcategory $\mathcal{A}={\rm add}\; (S_1\oplus S_2\oplus I[1])$.
Then $\mathcal{A}$ is a hereditary abelian category, which is even semisimple: it
is equivalent to the category of $H$-modules, where $H = k\times k \times k.$
Every indecomposable object of $\mathcal{D}$ can be shifted into $\mathcal{A}$, but
there is a nonzero homomorphism $A \rightarrow B[-1]$ where $A, B$ belong to
$\mathcal{A}$;  just take $A = S_2, B = I[1].$
Observe that the categories $\mathcal{D}$ and $\mathbf{D}^b(\mathcal{A})$ are not equivalent.}
\end{exm}

\section{Algebraical hereditary triangulated categories}

In this section, we prove that if the triangulated category is algebraical and hereditary, then it is triangle equivalent to the bounded derived category of a hereditary abelian category.  We use the existence result on the realization functor in \cite{Bei}.

\subsection{The triangle equivalence}

Let $(\mathcal{D}^{\leq 0}, \mathcal{D}^{\geq 0})$ be a bounded $t$-structure on a triangulated category $\mathcal{D}$ with the heart $\mathcal{H}$. By a \emph{realization functor} of the $t$-structure, we mean a triangle functor $F\colon \mathbf{D}^b(\mathcal{H})\rightarrow \mathcal{D}$ such that $F(\mathcal{H})=\mathcal{H}$ and its restriction $F|_\mathcal{H}$ is isomorphic to the identity functor. We observe that $F$ sends $\mathbf{D}^b(\mathcal{H})^{\leq 0}$ to $\mathcal{D}^{\leq 0}$, $\mathbf{D}^b(\mathcal{H})^{\geq 0}$ to $\mathcal{D}^{\geq 0}$.

\begin{lem}\label{lem:real}
Let $(\mathcal{D}^{\leq 0}, \mathcal{D}^{\geq 0})$ be a hereditary $t$-structure on $\mathcal{D}$. Then any realization functor $F\colon \mathbf{D}^b(\mathcal{H})\rightarrow \mathcal{D}$ is a triangle equivalence.
\end{lem}

\begin{proof}
The abelian category  $\mathcal{H}$ is hereditary by Lemma \ref{lem:split}. For the triangle equivalence, it suffices to show that $F$ is an equivalence. For the fully-faithfulness, by applying \cite[Lemma II.3.4]{Hap} it suffices to show that $F$ induces an isomorphism between ${\rm Ext}_\mathcal{A}^i(Z, X)$ and ${\rm Hom}_\mathcal{D}(Z, X[i])$ for each $i\in \mathbb{Z}$. The cases $i\leq 0$ are trivial, and the cases $i\geq 2$ are also trivial by the hereditary condition. The remaining case $i=1$ follows from the isomorphism (\ref{equ:iso1}), since the triangle functor $F$ necessarily sends $[\xi]$ to $\omega$.

The essential image ${\rm Im}\; F$ of $F$ is a triangulated subcategory of $\mathcal{D}$ containing $\mathcal{A}$. It is well known that the smallest triangulated subcategory of $\mathcal{D}$ containing $\mathcal{A}$ is  $\mathcal{D}$ itself; compare (\ref{equ:iso4}).  Then we have ${\rm Im}\; F=\mathcal{D}$, proving the denseness of $F$.
\end{proof}

Recall from \cite[Subsection 8.7]{Kel07} that a triangulated category is \emph{algebraical} provided that it is triangle equivalent to the stable category of a Frobenius category. For example, the bounded derived category of an essentially small abelian category is algebraical; see \cite[Subsection 7.7]{Kra}.

The following version of \cite[Subsection A.6]{Bei} seems to be quite convenient for application, which will be proved in the next subsection.

\begin{prop}\label{prop:real}
Let $\mathcal{D}$ be an algebraical triangulated category with a bounded $t$-structure and its heart $\mathcal{H}$. Then there is a realization functor $F\colon \mathbf{D}^b(\mathcal{H})\rightarrow \mathcal{D}$.
\end{prop}

By combing Theorem \ref{thm:A}, Lemma \ref{lem:real} and Proposition \ref{prop:real}, we obtain the promised triangle equivalence.  It generalizes \cite[Proposition 4.2]{BR},  where the triangulated category $\mathcal{D}$ is assumed to be  the bounded derived category of some abelian category.

\begin{thm}\label{thm:A.1}
Let $\mathcal{D}$  be  a hereditary triangulated category, which is algebraical. Then $\mathcal{D}$ is triangle equivalent to the bounded derived category of a hereditary abelian category. \hfill $\square$
\end{thm}

\begin{rem}
(1) The above assumption of being algebraical is natural. Indeed, if the  triangulated category $\mathcal{D}$ is essentially  small, then it is triangle equivalent to the bounded derived category of a hereditary abelian category if and only if it is algebraical and hereditary.

(2) The case described in Theorem \ref{thm:A.1} seems to be one of the
rare situations, where the derived categories of a class of abelian
categories can easily be characterized as special triangulated categories.
Usually, the axioms of a triangulated category tend to be too broad for such
an endeavour.
\end{rem}

It is natural to ask the following question: is there a non-algebraical hereditary triangulated category? In view of \cite[Subsection A.6]{Bei} and Lemma \ref{lem:real}, such an example will provide a triangulated category, over which there are no filtered triangulated categories. Recently, this question is answered in the negative in \cite{Hu}, where it is proved that  any hereditary triangulated category is algebraical.

\subsection{Filtered objects and the realization functor}

  We will show that the formalism in \cite[Appendix]{Bei} on the existence of a realization functor applies for an algebraical triangulated category, and then prove Proposition \ref{prop:real}.

   We will construct explicitly  a filtered triangulated category  over any algebraical triangulated category. We mention that the treatment here unifies  the one in \cite[Section 3.1]{BBD} and \cite[Subsection  2.5]{AR}.

Let $\mathcal{A}$ be an additive category. An \emph{exact pair} $(i, d)$ consists of two composable morphisms $X\stackrel{i}\rightarrow Y\stackrel{d}\rightarrow Z$ such that $i={\rm Ker}\; d$ and $d={\rm Cok}\; i$. An exact structure $\mathcal{E}$ on $\mathcal{A}$ is a class of exact pairs, which is closed under isomorphisms and satisfies certain axioms. The pair $(\mathcal{A}, \mathcal{E})$ is called an \emph{exact category} in the sense of Quillen. The exact pairs $(i, d)$ in $\mathcal{E}$ are called conflations, where $i$ are inflations and $d$ are deflations. When the exact structure $\mathcal{E}$ is understood, we will call $\mathcal{A}$ an exact category. For details on exact categories, we refer to \cite[Appendix A]{Kel90}.

The following consideration is inspired by \cite{Kel90, Chen, AR}. Let $\mathcal{A}$ be an exact category. A \emph{filtered object} in $\mathcal{A}$ is an infinite  sequence in $\mathcal{A}$
\begin{align*}
\cdots \longrightarrow  X_{n+1}\stackrel{i_{n+1}}\longrightarrow X_{n} \stackrel{i_n}\longrightarrow X_{n-1}\longrightarrow \cdots
\end{align*}
such that  each morphism $i_n$ is an inflation and that for sufficiently large $n$, $X_{n}=0$, $X_{-n}=X$ and $i_{-n}={\rm Id}_X$ for some object $X$.  This filtered object is denoted by $X_\bullet=(X_\bullet, i_\bullet)$ or $(X_\bullet, i_\bullet^X)$, where $X$ is called its \emph{underlying object}. We denote  by $i_{n, -\infty}\colon X_n\rightarrow X$ the canonical morphism for each $n\in \mathbb{Z}$. As a finite composition of inflations, this canonical morphism is an inflation.

A morphism $f_\bullet\colon (X_\bullet, i_\bullet^X)\rightarrow (Y_\bullet, i_\bullet^Y)$ between filtered objects consists of morphisms $f_n\colon X_n\rightarrow Y_n$ satisfying $i^Y_n\circ f_n=f_{n-1}\circ i^X_{n}$. The composition of morphisms is componentwise. Then we have the category $F\mathcal{A}$ of filtered objects; it is an additive category. We denote by
$$\omega\colon F\mathcal{A}\longrightarrow \mathcal{A}$$
 the \emph{forgetful functor}, which sends each filtered object to its underlying object.

Each object $A$ in $\mathcal{A}$ defines a filtered object $j(A)$ by $j(A)_n=0$ for $n>0$, $j(A)_n=A$ and $i_n^{j(A)}={\rm Id}_A$ for $n\leq  0$. This gives rise to an additive functor $$j\colon \mathcal{A} \longrightarrow F\mathcal{A,}$$
 which is fully faithful.

The proof of the following lemma is by  a routine verification; compare \cite[Subsection 5.1]{Kel90}.

\begin{lem}
The category $F\mathcal{A}$ of filtered objects has an exact structure such that the conflations are given by $X_\bullet\stackrel{f_\bullet}\longrightarrow Y_\bullet \stackrel{g_\bullet}\longrightarrow Z_\bullet$ with each pair $(f_n, g_n)$ a conflation in $\mathcal{A}$. Moreover, the forgetful functor $\omega\colon F\mathcal{A}\rightarrow \mathcal{A}$  and $j\colon \mathcal{A}\rightarrow F\mathcal{A}$ are exact. \hfill $\square$
\end{lem}

For a filtered object $(X_\bullet, i_\bullet)$, we define a new filtered object $s(X_\bullet, i_\bullet)$ by $s(X)_n=X_{n-1}$ and $i_n^{s(X)}=i_{n-1}$. This gives rise to an automorphism
$$s\colon F\mathcal{A}\longrightarrow F\mathcal{A}$$
of exact categories, called the \emph{filtration-shift functor}. We observe a natural transformation $\alpha\colon {\rm Id}_{F\mathcal{A}}\rightarrow s$ by $(\alpha_{(X_\bullet, i_\bullet)})_n=i_n$ for each $n\in \mathbb{Z}$.

 We denote by $F\mathcal{A}(\leq 0)$ the full subcategory of $F\mathcal{A}$ consisting of objects $(X_\bullet, i_\bullet)$ with  $X_{n}$=0 for each $n\geq 1$.  Similarly, the full subcategory $F\mathcal{A}(\geq 0)$ are formed by objects $(X_\bullet, i_\bullet)$ with $X_{-n}=X$ and $i_{-n}={\rm Id}_X$ for all $n\geq 0$, where $X$ is the underlying object. For $d\in \mathbb{Z}$, we set $F\mathcal{A}(\leq d)=s^dF\mathcal{A}(\leq 0)$ and $F\mathcal{A}(\geq d)=s^dF\mathcal{A}(\geq 0)$.

The following result is analogous to \cite[Lemma 2.7]{AR}.

\begin{lem}\label{lem:f}
The following statements hold.
\begin{enumerate}
\item $F\mathcal{A}(\leq 0)\subseteq F\mathcal{A}(\leq 1)$, $F\mathcal{A}(\geq 1)\subseteq F\mathcal{A}(\geq 0)$, and $F\mathcal{A}=\bigcup_{n\in \mathbb{Z}} F\mathcal{A}(\leq n)= \bigcup_{n\in \mathbb{Z}} F\mathcal{A}(\geq n)$.
\item We have $s(\alpha_{X_\bullet})=\alpha_{s(X_\bullet)}$ for each filtered object $X_\bullet$.
\item For any filtered objects $X_\bullet\in F\mathcal{A}(\geq 1)$ and $Y_\bullet\in F\mathcal{A}(\leq 0)$, we have ${\rm Hom}_{F\mathcal{A}}(X_\bullet, Y_\bullet)=0$. Moreover, we have an isomorphism
    \begin{align*}
    {\rm Hom}_{F\mathcal{A}}(s(Y_\bullet), X_\bullet) \stackrel{\sim}\longrightarrow     {\rm Hom}_{F\mathcal{A}}(Y_\bullet, X_\bullet)
    \end{align*}
    sending $f_\bullet$ to $f_\bullet \circ \alpha_{Y_\bullet}$.
\item Any filtered object $X_\bullet$ fits into a conflation $A_\bullet \rightarrow X_\bullet \rightarrow B_\bullet$ with $A_\bullet\in F\mathcal{A}(\geq 1)$ and $B_\bullet \in F\mathcal{A}(\leq 0)$.
    \item The functor $j$ induces an equivalence $j\colon \mathcal{A}\stackrel{\sim}\longrightarrow F\mathcal{A}(\leq 0)\cap F\mathcal{A}(\geq 0)$ of exact categories.
\end{enumerate}
\end{lem}

\begin{proof}
The statements (1), (2) and (3) are direct. We mention that in the isomorphism of (3), both the Hom groups are isomorphic to ${\rm Hom}_\mathcal{A}(\omega(Y_\bullet), \omega(X_\bullet))$. The statement (5) is direct, since the exact functor $j$ is fully faithful and reflects conflations.

For (4), we consider a filtered object $(X_\bullet, i_\bullet)$.  Set $A_{n}=X_n$ and $i_n^A=i_n$ for $n\geq 2$, $A_{n}=X_1$ and $i_n^A={\rm Id}_{X_1}$ for $n\leq 1$. For each $n<1$, we denote by $X_n/{X_1}$ the cokernel of the inflation $i_{n+1}\circ \cdots \circ i_0\circ i_1\colon X_{1}\rightarrow X_n$, and by $\bar{i}_n\colon X_n/{X_1}\rightarrow X_{n-1}/{X_1}$ the induced morphism of $i_n$, which is also an inflation. Set $B_n=0$ for $n\geq 1$, $B_{n}=X_n/{X_1}$ and $i_n^B=\bar{i}_n$ for $n\leq 0$. Then the canonical morphisms  $(A_\bullet, i_\bullet^A) \rightarrow X_\bullet$ and $X_\bullet \rightarrow (B_\bullet, i_\bullet^B)$ form the required conflation in $F\mathcal{A}$.
\end{proof}

The functor $j\colon \mathcal{A}\rightarrow F\mathcal{A}$ has a right adjoint and a left adjoint, both of which are exact. The functor $p\colon F\mathcal{A}\rightarrow \mathcal{A}$, which takes the zero component, is defined by $p(X_\bullet)=X_0$. We have the adjoint pair $(j, p)$ by the following natural isomorphism
 \begin{align}\label{equ:iso2}
 {\rm Hom}_{F\mathcal{A}} (j(A), X_\bullet) \stackrel{\sim}\longrightarrow {\rm Hom}_\mathcal{A}(A, p(X_\bullet))
 \end{align}
sending $f_\bullet$ to $f_0$. For a filtered object $(X_\bullet, i_\bullet)$ with its underlying object $X$, we consider the canonical inflation $i_{1, -\infty}\colon X_1\rightarrow X$, and set $c(X_\bullet)=X/X_1$ to be  its cokernel.  This gives rise to an additive functor $c\colon F\mathcal{A}\rightarrow \mathcal{A}$. The adjoint pair $(c, j)$ is given by the following natural isomorphism
 \begin{align}\label{equ:iso3}
 {\rm Hom}_{F\mathcal{A}} (X_\bullet, j(A)) \stackrel{\sim}\longrightarrow {\rm Hom}_\mathcal{A}(c(X_\bullet), A)
 \end{align}
sending $f_\bullet$ to the induced morphism $X/X_1\rightarrow A$ of $\omega(f_\bullet)\colon X\rightarrow A$. Here, we use the fact that $f_0\circ i_1=0$.

Recall that an exact category $\mathcal{A}$ is \emph{Frobenius} provided that it has enough projective objects and enough injective objects such that projective objects coincide with injective objects. We denote by $\underline{\mathcal{A}}$ the \emph{stable category} modulo projectives. For each object $X$, we fix a conflation $0\rightarrow X\stackrel{i_X} \rightarrow I(X)\stackrel{d_X}\rightarrow \mathbf{S}(X)\rightarrow 0$ with $I(X)$ injective. Then $\mathbf{S}$ yields an auto-equivalence $\mathbf{S}\colon \underline{\mathcal{A}}\rightarrow \underline{\mathcal{A}}$. The stable category $\underline{\mathcal{A}}$ has a canonical triangulated structure such that the translation functor is given by $\mathbf{S}$ and that exact triangles are induced by conflations. For details, we refer to \cite[Section I.2]{Hap}.

The proof of the following lemma is similar to \cite[Lemma 2.1]{Chen}.

\begin{lem}
Let $\mathcal{A}$ be a Frobenius category. Then the exact category $F\mathcal{A}$ is Frobenius. Moreover, a filtered object $X_\bullet$ is projective in $F\mathcal{A}$ if and only if each component $X_n$ is projective in $\mathcal{A}$.
\end{lem}

\begin{proof}
We observe that the functor $p$ is exact. It follows from the adjunction (\ref{equ:iso2})  that $j(P)$ is projective for any projective object $P$ in $\mathcal{A}$. Therefore, for each $d$,  $s^dj(P)$ is projective. For a filtered object $(X_\bullet, i_\bullet)$, there exist sufficiently large $a$ and $b$ such that $X_n=0$ for $n>a$, $X_n=X$ and $i_n={\rm Id}_X$ for $n\leq -b$. For each $-b\leq l \leq  a$, we denote by $X_l/X_{l+1}$ the cokernel of $i_{l+1}\colon X_{l+1}\rightarrow X_l$. Take a deflation $P_l\rightarrow X_l/{X_{l+1}}$ in $\mathcal{A}$ with $P_l$ projective. Thus we have a deflation $s^lj(P_l)\rightarrow s^lj(X_l/X_{l+1})$. We claim that there is a deflation $P_\bullet=\bigoplus_{l=-b}^{a} s^lj(P_l)\rightarrow (X_\bullet, i_\bullet)$ in $F\mathcal{A}$.

Indeed, there is a sequence of inflations in $F\mathcal{A}$
$$0=Y^{a+1}_\bullet\longrightarrow Y_\bullet^{a} \longrightarrow Y_\bullet^{a-1}\longrightarrow \cdots \longrightarrow   Y^{-b+1}_\bullet \longrightarrow Y^{-b}_\bullet= X_\bullet$$
 with each factor isomorphic to $s^lj(X_l/{X_{l+1}})$ for $l=a, a-1, \cdots, -b$. More precisely, we have $Y_n^l=X_l$ for $n\leq l$, and $Y_n^l=X_n$ for $n>l$. We apply repeatedly the argument in the Horseshoe Lemma to the deflations  $s^lj(P_l)\rightarrow s^lj(X_l/X_{l+1})$. Then we have the required deflation.

Similarly, using the exact functor $c$ and (\ref{equ:iso3}), we infer that for each projective object $P$ in $\mathcal{A}$, $s^dj(P)$ is injective. Moreover, each filtered object $X_\bullet$ fits into an inflation $X_\bullet\rightarrow P_\bullet$ with $P_\bullet$ a finite direct sum of objects of the form $s^dj(P)$. We are done by combining the above statements.
\end{proof}

The above lemmas allow us to apply the formalism in \cite[Appendix]{Bei} to an algebraical triangulated category.

\vskip 10pt

\noindent \emph{Proof of Proposition \ref{prop:real}.}\;  Take a Frobenius category $\mathcal{A}$ such that its stable category $\underline{\mathcal{A}}$ is triangle equivalent to $\mathcal{D}$. Consider the Frobenius category $F\mathcal{A}$ of filtered objects and its stable category $\underline{F\mathcal{A}}$. The above functors $j$, $s$ and $\omega$ are exact that send projective objects to projective objects. By \cite[Lemma I.2.8]{Hap} they induce the corresponding triangle functors between the stable categories. The stable version of Lemma \ref{lem:f} holds, where the conflation in (4) is replaced by an exact triangle and the equivalence in (5) induces a triangle equivalence. It follows that the stable category $\underline{F\mathcal{A}}$ is a filtered triangulated category over $\underline{\mathcal{A}}$ and thus over $\mathcal{D}$; see \cite[Definition A.1]{Bei}. Then the existence of the realization functor $F$ follows from \cite[Subsection A.6]{Bei}. \hfill $\square$

\section{Paths in triangulated categories}

Let  $\mathcal{D}$ be a triangulated category. We will assume from now on
that every object in $\mathcal{D}$ is a finite direct
sum of indecomposable objects. Denote by ${\rm Ind}\; \mathcal{D}$ a complete class of representatives of indecomposable objects. We emphasize that $\mathcal{D}$ is not assumed to have split idempotents.

\subsection{Paths and blocks}

 Let $X, Y$ be two indecomposable objects in $\mathcal{D}$.
A \emph{path} of length $n$ is a sequence $X_0, X_1, \dots , X_n$ of
indecomposable objects in $\mathcal{D}$ such that
for $1 \leq i \leq n$, we have
${\rm Hom}_\mathcal{D}(X_{i-1},X_i) \neq 0$ or $X_i = X_{i-1}[1]$. We will say that the path \emph{starts} at $X_0$ and \emph{ends} in $X_n$, or that it is a path from $X_0$ to $X_n$.

A subclass $\mathcal{U}\subseteq {\rm Ind}\; \mathcal{D}$ is called  \emph{path-closed} provided that for each path from $X$ to $Y$, $X$ lies in $\mathcal{U}$ if and only if so does $Y$. Equivalently, the class $\mathcal{U}$ is closed under the translation functors $[1]$ and $[-1]$, and if for any $X\in \mathcal{U}$, an indecomposable object $Y$ necessarily lies in $\mathcal{U}$ whenever ${\rm Hom}_\mathcal{D}(X, Y)\neq 0$ or ${\rm Hom}_\mathcal{D}(Y, X)\neq 0$. We observe if $\mathcal{U}$ is path-closed, so is the complement $\mathcal{V}={\rm Ind}\; \mathcal{D}\backslash \mathcal{U}$.

A subclass $\mathcal{U}\subseteq {\rm Ind}\; \mathcal{D}$ is called  \emph{path-connected} provided that any two indecomposable objects in $\mathcal{U}$ are connected by a sequence of paths and inverse paths. More precisely, for each pair $X, Y$ of objects in $\mathcal{U}$, there exists a sequence $X=X_0, X_1, \cdots, X_t=Y$ of indecomposable objects such that for $1\leq i\leq t$, at least one of the three conditions ${\rm Hom}_\mathcal{D}(X_{i-1}, X_i)\neq 0$, ${\rm Hom}_\mathcal{D}(X_i, X_{i+1})\neq 0$, or $X_i=X_{i-1}[s]$ for some $s\in \mathbb{Z}$, is satisfied.

\begin{lem}\label{lem:decomposition}
Let $\mathcal{U}\subseteq {\rm Ind}\; \mathcal{D}$ a path-closed class and let $\mathcal{V}$ its complement. Set $\mathcal{D}_1={\rm add}\; \mathcal{U}$ and $\mathcal{D}_2={\rm add}\; \mathcal{V}$. Then both $\mathcal{D}_i$ are triangulated subcategories, and $\mathcal{D}=\mathcal{D}_1\times \mathcal{D}_2$ is their product.
\end{lem}

\begin{proof}
By their path-closedness, we have ${\rm Hom}_\mathcal{D}(\mathcal{D}_1, \mathcal{D}_2)=0={\rm Hom}_{\mathcal{D}}(\mathcal{D}_2, \mathcal{D}_1)$. Then we have the decomposition $\mathcal{D}=\mathcal{D}_1\times \mathcal{D}_2$ of additive categories. We observe that both $\mathcal{D}_i$ are closed under $[1]$ and $[-1]$. To complete the proof, we take a morphism $u\colon A\rightarrow B$ in $\mathcal{D}_1$ and form an exact triangle $A\stackrel{u}\rightarrow B\rightarrow C\rightarrow A[1]$. We assume that $C=C_1\oplus C_2$ with $C_i\in \mathcal{D}_i$. Then $C_2=0$ by Lemma \ref{lem:basic}(1). This proves that $\mathcal{D}_1$ is a triangulated subcategory.
\end{proof}

The triangulated category $\mathcal{D}$ is called a \emph{block} provided that it is nonzero and does not admit a proper decomposition into two triangulated subcategories.

\begin{prop}\label{prop:block}
 Let $\mathcal{D}$ be a nonzero triangulated category. Then the following statements are equivalent:
\begin{enumerate}
\item The triangulated category $\mathcal{D}$ is a block.
 \item The only  nonempty path-closed subset $\mathcal{U}\subseteq {\rm Ind}\; \mathcal{D}$ is  ${\rm Ind}\; \mathcal{D}$ itself.
     \item The set ${\rm Ind}\; \mathcal{D}$ is path-connected.
\end{enumerate}
\end{prop}

\begin{proof}
The implication ``$(1)\Rightarrow (2)$" follows from Lemma \ref{lem:decomposition}. For ``$(2)\Rightarrow (3)$", let $X$ be an indecomposable object. Denote by $\mathcal{X}\subseteq {\rm Ind}\; \mathcal{D}$ the class formed by those objects $Y$, which can be connected to $X$ by a sequence of paths and inverse paths. We observe that $\mathcal{X}$ is path-closed. Then we have $\mathcal{X}={\rm Ind}\; \mathcal{D}$. This proves that ${\rm Ind}\; \mathcal{D}$ is path-connected.

To prove ``$(3)\Rightarrow (1)$", we assume on the contrary that $\mathcal{D}=\mathcal{D}_1\times \mathcal{D}_2$ is a proper decomposition. Both $\mathcal{D}_i$'s contain indecomposable objects. Take indecomposable objects $X\in \mathcal{D}_1$ and $Y\in \mathcal{D}_2$. Then there are no sequences which connect $X$ with $Y$ and consist  of paths and inverse paths. This contradicts to the path-connectedness.
\end{proof}

For a division ring $D$ and $n\geq 1$, consider the direct product $D^n=D\times \cdots \times D$ of $n$ copies of $D$. The module category $D^n\mbox{-mod}$ is semisimple. An automorphism $\sigma$ on $D$ yields an automorphism $\sigma^n\colon D^n\rightarrow D^n$ sending $(x_1, x_2, \cdots, x_n)$ to $(x_2, \cdots, x_n, \sigma(x_1))$. We denote by $(\sigma^n)^*\colon D^n\mbox{-mod}\rightarrow D^n\mbox{-mod}$ the automorphism of twisting the $D^n$-action by $\sigma^n$.

Recall from \cite[Lemma 3.4]{Chen11} that any semisimple abelian category $\mathcal{A}$ becomes a triangulated category with the translation functor being any prescribed auto-equivalence $\Sigma$ on $\mathcal{A}$. The exact triangles are direct sums of trivial ones. The resulted triangulated category is denoted by $(\mathcal{A}, \Sigma)$. In particular, we have the triangulated category $(D^n\mbox{-mod}, (\sigma^n)^*)$.

We say that a block $\mathcal{D}$ is \emph{degenerate} provided that there is an indecomposable object $X$ satisfying the condition: any nonzero morphisms $Y\rightarrow X$ and $X\rightarrow Y'$, with $Y, Y'$ indecomposable, are invertible.

\begin{lem}\label{lem:dege}
Let $\mathcal{D}$ be a degenerate block with the above indecomposable object $X$. Then the following statements hold.
\begin{enumerate}
\item ${\rm Ind}\; \mathcal{D}=\{X[s]\; |\; s\in \mathbb{Z}\}$ and ${\rm End}_\mathcal{D}(X)=D^{\rm op}$, where $D$ is a division ring.
\item If $X$ is not isomorphic to $X[s]$ for each $s>0$, then there is a triangle equivalence $\mathcal{D}\stackrel{\sim}\longrightarrow \mathbf{D}^b(D\mbox{-{\rm mod}})$.
\item If $n$ is the smallest natural number such that $X$ is isomorphic to $X[n]$, then there is a triangle equivalence $\mathcal{D}\stackrel{\sim}\longrightarrow (D^n\mbox{-{\rm mod}}, (\sigma^n)^*)$ for some automorphism $\sigma$ of $D$.
\end{enumerate}
\end{lem}

\begin{proof}
From the assumption, we observe that $\mathcal{U}=\{X[s]\; |\; s\in \mathbb{Z}\}$ is path-closed. Then (1) follows from Proposition \ref{prop:block}(2) immediately. The equivalences in (2) and (3) are evident. We omit the details. We mention that in (3), the automorphism $\sigma$ on $D$ is induced by the action of $[n]$ on morphisms in $\mathcal{D}$.
\end{proof}

\subsection{The existence of paths} We will study paths in a non-degenerate block. We keep the assumption that  in the triangulated category $\mathcal{D}$, any object is a finite direct sum of indecomposable objects.

\begin{lem}\label{lem:often}
Let $u\colon X\rightarrow Y$ be a nonzero non-invertible morphism between indecomposable objects in $\mathcal{D}$. Then there is an exact triangle $$X\stackrel{u}\longrightarrow Y \stackrel{\binom{*}{v}}\longrightarrow Z'\oplus Z \stackrel{(*, w)}\longrightarrow X[1]$$
such that $Z$ is indecomposable and that both morphisms $v$ and $w$ are nonzero non-invertible.
\end{lem}

\begin{proof}
Since $u$ is non-invertible, its cone is not zero. Since $\binom{*}{v}\circ u=0$ and $u[1]\circ (*, w)=0$, we infer that both $v$ and $w$ are non-invertible. Since $u\neq 0$, Lemma \ref{lem:basic}(2) implies that $v\neq 0$. By a dual argument, we have $w\neq 0$.
\end{proof}

We also observe the dual of Lemma \ref{lem:often}.

\begin{lem}\label{lem:dual-often}
Let $u\colon X\rightarrow Y$ be a nonzero non-invertible morphism between indecomposable objects in $\mathcal{D}$. Then there is an exact triangle
$$Y[-1] \stackrel{\binom{*}{v}}\longrightarrow Z'\oplus Z \stackrel{(*, w)}\longrightarrow X\stackrel{u}\longrightarrow Y$$
such that $Z$ is indecomposable and that both morphisms $v$ and $w$ are nonzero non-invertible. \hfill $\square$
\end{lem}

\begin{lem}\label{lem:mor}
 Let $\mathcal{D}$ be a triangulated category which is a non-degenerate block.
Then for any indecomposable object $X$, there is a sequence
$X = X_0, X_1, X_2, X_3 = X[1]$ of indecomposable objects with ${\rm Hom}_\mathcal{D}(X_{i-1},X_i) \neq 0.$
\end{lem}

\begin{proof}
By the non-degeneration of $\mathcal{D}$, we  assume that there exists
a nonzero and non-invertible morphism $u\colon X\rightarrow Y$ or $u\colon Y\rightarrow X$ with $Y$ indecomposable. In the first case, we apply Lemma \ref{lem:often} and then the morphisms $u, v, w$ yield the required sequence. In the second case, we apply Lemma \ref{lem:dual-often} to obtain a sequence from $X[-1]$ to $X$. Applying $[1]$ to this sequence, we are done.
\end{proof}

\begin{rem}\label{rem:1}
The following immediate consequence of Lemma \ref{lem:mor} is of interest. In a non-degenerate block $\mathcal{D}$, any path from $X$ to $Y$ can be refined to a path $X=X_0, X_1, \cdots, X_t=Y$ such that ${\rm Hom}_\mathcal{D}(X_{i-1},X_i) \neq 0$  for $1\leq i\leq t$.
\end{rem}

\begin{lem}\label{lem:inverse}
Let $\mathcal{D}$ be a triangulated category which is a block. Assume that there is a path from $X$ to $Y$. Then there is a path from $Y$ to $X[n]$ for some $n\geq 0$.
\end{lem}

\begin{proof}
If $\mathcal{D}$ is degenerate, the statement is immediate by Lemma \ref{lem:dege}(1).  We assume that $\mathcal{D}$ is non-degenerate.

We first prove that if there is a nonzero morphism $u\colon X\rightarrow Y$, then there is a path from $Y$ to $X[1]$. Indeed, if $u$ is invertible, there is nothing to prove. Otherwise, we use the morphisms $u, v, w$ in Lemma \ref{lem:often} to obtain the required path.

In general, by Remark \ref{rem:1} we assume that there is a path $X=X_0, X_1, \cdots, X_t=Y$ with ${\rm Hom}_\mathcal{D}(X_{i-1},X_i) \neq 0$  for $1\leq i\leq t$. By the above argument, we have paths from $X_i$ to $X_{i-1}[1]$. By applying the translation functors and gluing the paths, we obtain a path from $Y$ to $X[t]$.
\end{proof}

The following result claims the existence of certain paths in a block.

\begin{prop}\label{prop:path}
Let $\mathcal{D}$ be a triangulated category which is a block. Let $X, Y$ be indecomposable objects in $\mathcal{D}$. Then there exists a path from $X$ to $Y[n]$ for some $n\geq 0$.
\end{prop}

\begin{proof}
By the path-connectedness, there is a sequence $X=X_0, X_1, \cdots, X_t=Y$ such that for $1\leq i\leq t$ there is a path from $X_{i-1}$ to $X_i$, or a path from $X_i$ to $X_{i-1}$. In the latter case, applying Lemma \ref{lem:inverse},  we have a path from $X_{i-1}$ to $X_i[m]$ for some $m\geq 0$. We now adjust the given sequence as $$X_0, \cdots, X_{i-1}, X_i[m], X_{i+1}[m], \cdots, X_t[m].$$ Repeating this procedure, we obtain the required path.
\end{proof}

Recall that we do not assume that the triangulated category $\mathcal{D}$ has split idempotents. However, the following observation implies that nontrivial idempotents on indecomposable objects  lead to some unexpected paths; compare Remark \ref{rem:B}(2).

\begin{lem}\label{lem:non-idem}
Let $X$ be an indecomposable object in $\mathcal{D}$ with a nontrivial idempotent $e\colon X\rightarrow X$. Then there is a path of length two from $X[1]$ to $X$.
\end{lem}

\begin{proof}
We form the exact triangle $X\stackrel{e}\rightarrow X \stackrel{u}\rightarrow C \stackrel{v}\rightarrow X[1]$.  There exist morphisms $a\colon X[1]\rightarrow C$ and $b\colon C\rightarrow X$ such that ${\rm Id}_C=u\circ b+a\circ v$ (this can be proved in the idempotent completion  \cite{BS} of $\mathcal{D}$, where $e$ equals ${\rm Id}_A\oplus 0\colon A\oplus B\rightarrow A\oplus B$ for some objects $A$ and $B$ in the idempotent completion.)

Write $C = \bigoplus_{i=1}^s C_i$ as a direct sum of
indecomposable objects $C_i$ in $\mathcal{D}$, thus $a = (a_1,\dots,a_s)^t$
with $a_i\colon X[1] \rightarrow C_i$ and $b = (b_1,\dots,b_s)$ with
$b_i\colon C_i \rightarrow  X$. We can assume that $a_i \neq 0$
if and only if $1\leq i \le r$. If there is $i$ with
$1\leq i \leq r$ such that also $b_i\neq 0$, then
we obtain a path  $X[1] \to C_i \to X$ of length two, as we want to show.

Assume now that no such path exists. Let $C' = \bigoplus_{i=1}^r C_i$
and $C'' = \bigoplus_{i=r+1}^t C_i$, let $a = (a',0)^t$
with $a'\colon X[1] \rightarrow C'$ and $b = (0,b'')$ with $b''\colon C'' \rightarrow X.$
Also, write $u = (u',u'')^t$ and $v = (v',v'')$. Then we have
$a'\circ v' = {\rm Id}_{C'}$ and $u''\circ b'' = {\rm Id}_{C''}.$
First, assume that $C' \neq 0.$ The equality
$a'\circ v' = {\rm Id}_{C'}$ shows that $v'$ is split mono. Since
$X[1]$ is indecomposable, it follows that $v'$ is an isomorphism
and therefore $v$ is split epic. This implies that $e = 0$,
a contradiction. Thus, we can assume that $C' = 0$. But this
implies that $u=u'' $ is split epic,  and therefore $e$ is split mono, thus $e = {\rm Id}_X$,
again a contradiction.
\end{proof}

\section{Hereditary triangulated categories which are blocks}

In this section, using the non-existence of certain paths, we characterize hereditary triangulated categories which are blocks.

Throughout, $\mathcal{D}$ is a triangulated category, in which each object is a finite direct sum of indecomposable objects.

 For an indecomposable object $X$, denote by $[X \to]$ the class
of all indecomposable objects $U$ in $\mathcal{D}$ with a path from $X$ to $U$. Then  $[X\to]$ is closed under the translation functor $[1]$. The complement of $[X\to]$ in ${\rm Ind}\; \mathcal{D}$ is closed under $[-1]$.

\begin{thm}\label{thm:B}
Let $\mathcal{D}$  be a triangulated category which is a block.
Then the following conditions are equivalent:
\begin{enumerate}
\item The triangulated category $\mathcal{D}$ is hereditary.
\item If $X$ is indecomposable in $\mathcal{D}$, then there is no path from $X[1]$ to $X$.
\item There is an indecomposable object $X$ in $\mathcal{D}$ with no path from $X[1]$ to $X$.
\item There are indecomposable objects $X, Y$ in $\mathcal{D}$ with no path from $Y$ to $X$.
\end{enumerate}
\end{thm}

\begin{proof}
The implication ``$(1)\Rightarrow (2)$"  is obvious. By Theorem \ref{thm:A}(3), we identify  $\mathcal{D}$ with $\mathbf{D}^b(\mathcal{H})$ for a hereditary abelian category $\mathcal{H}$. Since $X$ is indecomposable, it is of the form $X = A[n]$ for some indecomposable object $A\in \mathcal{H}$ and $n\in \mathbb{Z}$. By induction on the length of paths, we observe that $[X\to] \subseteq \bigcup_{i\ge n} \mathcal{H}[i]$. In particular, we have that $X[-1]$ does not belong to $[X\to]$.

The implications ``$(2)\Rightarrow (3)$" and ``$(3)\Rightarrow (4)$" are trivial. For ``$(4)\Rightarrow (3)$", we consider the given indecomposable objects $X, Y$. By Proposition \ref{prop:path}, there is a path from $Y$ to $X[n]$ for some $n\geq 0$. By assumption, we infer that $n\geq 1$. If there is a path from $X[1]$ to $X$, we obtain a path from $X[n]$ to $X$. This yields a path from $Y$ to $X$, a contradiction.

It remains to show ``$(3)\Rightarrow (1)$". Write $\mathcal{U}=[X\to]$ and $\mathcal{V}={{\rm Ind}\; \mathcal{D}}\backslash \mathcal{V}$ its complement. Let
$$\mathcal{A}={\rm add}\; (\mathcal{U}\cap \mathcal{V}[1]).$$
We will prove that $\mathcal{A}$ satisfies the conditions in Theorem \ref{thm:A}(2). Then we are done.

\emph{Step 1}\; For a nonzero morphism $u\colon A\rightarrow B$ between indecomposable objects with $A\in \mathcal{A}$, we observe that $B\in \mathcal{U}$. We claim that $B\notin \mathcal{U}[2]$.

From the claim, we infer that $B$ lies in $\mathcal{A}$ or $\mathcal{A}$[1]. Indeed,  if $B\in \mathcal{V}[1]$, we have $B\in \mathcal{A}$. Otherwise, we have $B\in \mathcal{U}[1]$ and by the claim, $B\in \mathcal{V}[2]$. Hence, we have $B\in \mathcal{A}[1]$.

To prove the claim, we assume on the contrary that $B\in \mathcal{U}[2]$. Then we have a path from $X[1]$ to $B[-1]$. By the facts that $A\notin \mathcal{U}[1]$ and $\mathcal{U}[2]\subseteq \mathcal{U}[1]$, we infer that $u$ is not an isomorphism. We obtain by Lemma \ref{lem:dual-often} a path of length two from $B[-1]$ to $A$. Then we have a path from $X[1]$ to $A$, that is, $A\in \mathcal{U}[1]$. A contradiction!

\emph{Step 2}\; To show $\mathcal{D}={\rm add}\; (\bigcup_{n\in \mathbb{Z}} \mathcal{A}[n])$, we claim that each indecomposable object $Y\in \mathcal{D}$ is of the form $B[m]$ for some $B\in \mathcal{A}$ and $m\in \mathbb{Z}$.

We observe by assumption that $X\in \mathcal{A}$. Assume first that $Y\in \mathcal{U}=[X\to]$. Then there is a path $X=X_0, X_1, \cdots, X_t=Y$. By induction on the length of paths, we may assume that $X_{t-1}=A[n]$ for $A\in \mathcal{A}$ and some $n\in \mathbb{Z}$. If $Y=X_{t-1}[1]$, then we are done by $Y=A[n+1]$. If ${\rm Hom}_\mathcal{D}(X_{t-1}, Y)\neq 0$, equivalently ${\rm Hom}_\mathcal{D}(A, Y[-n])\neq 0$, we infer from \emph{Step 1} that $Y[-n]$ lies in $\mathcal{A}$ or $\mathcal{A}[1]$. This also proves the statement in this case.

For the general case, by Proposition \ref{prop:path} there is a path from $X$ to $Y[d]$ for some $d\geq 0$, that is, $Y[d]\in [X \to]$. Applying the above argument to $Y[d]$, we are done for the claim.

\emph{Step 3}\; We claim that ${\rm Hom}_\mathcal{D}(\mathcal{A}, \mathcal{A}[m])=0$ for $m<0$. We assume the contrary. Take two indecomposable objects $A, B\in \mathcal{A}$ with ${\rm Hom}_\mathcal{D}(A, B[m])\neq 0$. Then $B[m]$ lies in $\mathcal{U}$. On the other hand, $B[m]$ lies in $\mathcal{V}[m+1]$. Since $m+1\leq 0$, we have $\mathcal{V}[m+1]\subseteq \mathcal{V}$. We conclude that $B[m]\in \mathcal{U}\cap \mathcal{V}$, a contradiction! This completes the whole proof.
\end{proof}

\begin{rem}\label{rem:B}
(1) The equivalence of the conditions in Theorem \ref{thm:B}(2) and (3) is somehow surprising: the existence of a single indecomposable object with a special property forces all the indecomposable objects to have this property! This indicates a rather unusual character of homogeneity.

(2) It is well known that the bounded derived category of an abelian category has split idempotents. In particular, a hereditary triangulated category has split idempotents. Then Theorem \ref{thm:B}(3) allows us to  strengthen Lemma \ref{lem:non-idem}. Assume that the block $\mathcal{D}$  does not have split idempotents. Then there are paths from $X[1]$ to $X$ for any idempotent object $X$ in $\mathcal{D}$.
\end{rem}

In the remaining part, we draw some immediate consequences of Theorem \ref{thm:B}.

Let us call a path $X_0, X_1, \dots , X_n$ \emph{proper} provided that
for $1 \leq i \leq n$, there exists a nonzero and non-invertible map
$X_{i-1} \to X_i$ or else $X_i = X_{i-1}[1]$. An indecomposable object $X$ in a triangulated category $\mathcal{D}$ will called \emph{directing} provided there is no proper path of length at least one starting and ending in $X$.

The corresponding notion of \emph{directing objects} in an abelian category is well known and has been found useful in \cite{Rin}, where the \emph{paths}  are of the form $X_0, X_1, \dots , X_n$ such that
for $1 \leq i \leq n$, there exists a nonzero and non-invertible map $X_{i-1} \to X_i$. The following observation is immediate: for a hereditary abelian category $\mathcal{H}$, an indecomposable object $X$ in $\mathcal{H}$ is directing if and only if $X$ is directing in $\mathbf{D}^b(\mathcal{H})$.

Hereditary abelian categories with directing objects are studied in \cite{HR}. The following result characterizes their bounded derived categories.

\begin{prop}\label{prop:chara}
Let $\mathcal{D}$ be an algebraical triangulated category which is a block. Then the following two statements are equivalent:
\begin{enumerate}
\item The triangulated category $\mathcal{D}$ has a directing object.
\item There is a triangle equivalence $\mathbf{D}^b(\mathcal{H})\rightarrow \mathcal{D}$, where $\mathcal{H}$ is a hereditary abelian category with a directing object.
\end{enumerate}
\end{prop}

\begin{proof}
The implication ``$(2)\Rightarrow (1)$" is already indicated by the above discussion. For the converse, let $X$ be a directing object in $\mathcal{D}$. Any path
from $X[1]$ to $X$ could be composed with the path $X, X[1]$. After deleting some repetitions, we obtain a proper path from $X$ to $X$ of length at least one. Thus, no path from $X[1]$ to $X$ exists. We are done by Theorems \ref{thm:B} and  \ref{thm:A.1}.
\end{proof}

Let $k$ be a field. Recall from \cite{HRS2} that a finite dimensional $k$-algebra $A$ is \emph{piecewise hereditary} provided that the bounded derived category $\mathbf{D}^b(A\mbox{-mod})$ of the module category is triangle equivalent to $\mathbf{D}^b(\mathcal{H})$ for a
hereditary abelian category $\mathcal{H}$. If $k$ is algebraically closed, such a hereditary abelian category $\mathcal{H}$ is derived equivalent to the module category over a path algebra or the category of coherent sheaves on a weighted projective line; see \cite{Hap2}.

It is well known that $\mathbf{D}^b(A\mbox{-mod})$ is an algebraical triangulated category. Theorems \ref{thm:B} and  \ref{thm:A.1}  yield the following characterization of piecewise hereditary algebras.

\begin{cor}
Let $A$ be a finite dimensional connected $k$-algebra. Then the following conditions are equivalent:
\begin{enumerate}
\item The algebra $A$ is piecewise hereditary.
\item For any indecomposable object $X$ in $\mathbf{D}^b(A\mbox{-{\rm mod}})$, there is no path from $X[1]$ to $X$.
\item There exists an indecomposable object $X$ in $\mathbf{D}^b(A\mbox{-{\rm mod}})$ with no path from $X[1]$ to $X$. \hfill $\square$
\end{enumerate}
\end{cor}

We mention that in \cite{HZ}, the characterization of piecewise hereditary algebras in terms of finite strong global dimension relies on the above result.

 We observe the following immediate consequence of  combining Proposition \ref{prop:chara} and \cite[Theorem]{HR}.

\begin{cor}
Let $A$ be a finite dimensional connected $k$-algebra. Then  $\mathbf{D}^b(A\mbox{-{\rm mod}})$ contains a directing object if and only  if  $A$ is derived equivalent to a finite dimensional hereditary algebra. \hfill $\square$
\end{cor}

\vskip 10pt

\noindent {\bf Historical  Remarks and Acknowledgements}\quad  This paper combines two different investigations. The second part of the paper
is based on the preprint [C.M. Ringel, {\em Hereditary triangulated categories},
SFB-preprint {\bf 98-107}, 1998],  which was accepted at that time by the journal
Compositio. Its aim was to provide a characterization of the bounded derived category of a
hereditary abelian category, but the corresponding result (Theorem \ref{thm:B})
dealt only with equivalences of additive categories with a shift functor
and not with triangle equivalences. This gap was pointed out by Michel Van
den Bergh, thus the author never handed in a final version for publication. The gap was recently solved for triangulated categories which are algebraical by Chen. This comprises the first part of the paper.

The authors have decided to integrate both parts into the current paper
and they want to thank Michel Van den Bergh for suggesting the algebraical
assumption in Theorem \ref{thm:A.1}. The author thank Andrew Hubery for the reference \cite{Hu}, and Chao Zhang for pointing out some misprints.

Chen is supported by National Natural Science Foundation of China (No. 11522113) and the Fundamental Research Funds for the Central Universities.

\vskip 10pt

 {\footnotesize \noindent Xiao-Wu Chen \\
Key Laboratory of Wu Wen-Tsun Mathematics, Chinese Academy of Sciences\\
School of Mathematical Sciences, University of Science and Technology of China\\
No. 96 Jinzhai Road, Hefei, Anhui Province, 230026, P. R. China.\\
E-mail: xwchen$@$mail.ustc.edu.cn.}
 \vskip 5pt

 {\footnotesize \noindent Claus Michael Ringel\\
Faculty of Mathematics, University of Bielefeld\\
D-33501, Bielefeld, Germany.\\
 E-mail:ringel$@$math.uni-bielefeld.de.}
\end{document}